\documentclass[a4paper,10pt]{article}

\usepackage{times}
\usepackage[T1]{fontenc}
\usepackage[latin1]{inputenc}
\usepackage{amsmath,amssymb}
\usepackage{amsthm}
\usepackage{amscd}

\usepackage{graphicx}
\usepackage{amssymb}
\usepackage{amsmath}
\usepackage{epstopdf}

\newtheorem{theorem}{Theorem}[section]
\newtheorem{lemma}[theorem]{Lemma}
\newtheorem{corollary}[theorem]{Corollary}
\newtheorem{proposition}[theorem]{Proposition}

\newtheorem{definition}[theorem]{Definition}

\theoremstyle{definition}
\newtheorem{example}[theorem]{Example}
\newtheorem{remark}[theorem]{Remark}

\numberwithin{table}{section}
\numberwithin{equation}{section}

\begin{document}
\title{Rational functions as new variables } 
\author{ Diana Andrei, Olavi Nevanlinna, Tiina Vesanen}
\maketitle

 \begin{center}
{\footnotesize\em 
Aalto University\\
Department of Mathematics and Systems Analysis\\
 email: Diana.Andrei\symbol{'100}aalto.fi, \  Olavi.Nevanlinna\symbol{'100}aalto.fi, \ Tiina.Vesanen\symbol{'100}aalto.fi\\[3pt]
}
\end{center}

\begin{abstract}
In {\it multicentric calculus} one takes a polynomial $p$  with distinct roots as a new variable   and represents  complex valued functions by $\mathbb C^d$-valued functions, where $d$ is the degree of $p$.   An application is e.g. the possibility to represent a  piecewise constant holomorphic function as a convergent power series, simultaneously in all components of $|p(z)| \le \rho$.  In this paper we study the necessary modifications needed, if we  take  a rational function $r=p/q$ as the new variable instead.  This allows to consider functions defined in neighborhoods of any compact set  as opposed to the polynomial case where the domains $|p(z)| \le \rho$ are always polynomially convex.   Two applications are formulated. One giving a convergent power series expression for Sylvester equations $AX-XB =C$ in the general case of $A,B$ being bounded operators in Banach spaces with distinct spectra.  The other application formulates a  K-spectral result for bounded operators in Hilbert spaces.

 \end{abstract}
\bigskip
 


\bigskip

 \section{Introduction}
 
 In  a series of papers [7] - [10]  one of us has considered  the possibility and  applications of  taking a polynomial  with simple zeros as a new {\it global} variable $w=p(z)$.  As the polynomial of degree $d$ is  not  one-to-one, complex valued  scalar functions $\varphi$ are represented by $\mathbb C^d$-valued functions $f$.   Additionally [2] contains modifications to  the case $w=p(z)^n$ and [1] discusses extensions to n-tuples of operators. 
 
 The key idea  in  applications to functional calculus is to  have a polynomial $p$ such that $p(A)$ is either small so that the series expansions of $f$ converge fast at $p(A)$,  or "structurally simpler" than $A$   so that, for example,  a matrix $A$ with nontrivial Jordan blocks becomes diagonalizable.  
 
By Hilbert's lemniscate theorem, see  e.g. [11],  any  {\it polynomially convex} compact set can be approximated from  outside arbitrarily well using polynomial lemniscates $|p(z)| = \rho$.  Taking such a polynomial as a new variable maps the analysis from inside the lemniscate into a disc, where  a lot of analysis tools are available.  At the end one  transforms the results back into scalar functions in the original variable. 

Sometimes one needs to have a representation for a function in sets which are not polynomially convex.   To that  end it is natural to ask  whether taking a rational function in place of the polynomial  leads to a useful representation  in such cases.  
It turns out that  choosing a rational function $r=p/q$ with $q$ of lower degree than $p$ much of the multicentric calculus carries over with minor modifications.  

The paper is organized as follows.   We shall first  formulate and prove a "rational  lemniscate lemma" approximating  any compact set arbitrarily well {\it in a fixed neighborhood of it}.  This is done in Section 2 .  We also formulate a result  as corollary  where  the spectra of bounded operators play the role of the compact  set. 

In Section 3 we consider the existence and uniqueness of  the representations using rational functions as variables. 
Given $\varphi$ there exists a unique representing function $f$  excluding critical points of the rational function and  if  $\varphi$ is holomorphic then the singularities of $f$  at  critical values are removable, so $f$ is holomorphic as well.  In order  to determine   the  largest class of functions  for which the representation is continuous at  critical values we modify the approach in [9] by moving the focus into the functions  $f$ and construct a  unital Banach algebra for such functions so that the  original function $\varphi$  appears as the Gelfand transform of  $f$,  bringing the wealth of   Banach algebra theory  available.   Throughout the paper we use the convention that by  function holomorphic in a compact set  we mean that it is holomorphic in some unknown neighborhood of it.

We shall  indicate  two applications in  which  we remove the assumption on the compact set  to be polynomially convex, needed when using  polynomials as new variables.   Sylvester equation $AX-XB=C$   with bounded operators in Banach spaces, has a unique solution for every $C$ if  and only if  the spectra are separated: $\sigma(A) \cap \sigma(B)= \emptyset$. We show that then, without any other assumptions, there exists a rational function such that  the solution to the Sylvester equation can be represented as a convergent power series. This is discussed in Section 4.   This generalizes a result of [10] where a similar statement was shown for polynomials with the extra assumption that the polynomial convex hulls  of the spectra do not intersect.

In [8]  it was shown that  polynomial lemniscates provide $K$-spectral sets and we generalize the discussion in Section 5  for rational lemniscates.

 \section{Rational lemniscate sets}
 
 \subsection{Approximating compact sets with rational lemniscates}
 
Hilbert Lemniscate Theorem, e.g. [11], provides the existence of a polynomial such that it can surround any polynomially compact set arbitrarily closely.  In fact, given a compact $K$ such that $\mathbb C \setminus K$ is {\it simply connected } and $\varepsilon>0$ there exists a polynomial $p$ such that
if  
$$ V_p= \{ z \in \mathbb C :  |p(z)| \le 1 \}
$$
and $K_\varepsilon = \{ z  :  {\rm dist} (z,K) \le \varepsilon\}$, then
$$
K \Subset V_p   \Subset  K_\varepsilon.
$$
Here  $ \Subset $  means that the  smaller compact is   included in the interior of the larger compact.

Suppose $r=p/q$ is a rational function, with $p$ and $q$ having no common roots.  Again we put
\begin{equation}\label{ratlemma}
V_r= \{ z \in \mathbb C :  |r(z)| \le 1 \}
\end{equation}
but we need to restrict $V_r$ into a compact set as we do not have control of  the size of $r$ globally.  To that end we denote by $\Gamma_\varepsilon$ the following compact set surrounding $K$:
\begin{equation}\label{beltaround}
\Gamma_\varepsilon  := K_\varepsilon \setminus {\rm int } \ K_{\varepsilon/2}.
\end{equation}

\begin{theorem} Given a compact $K\subset \mathbb C$ and $\varepsilon >0$  let $\Gamma_\varepsilon$ be as in (\ref{beltaround}). Then   there exists a rational function $r$ such that
\begin{equation} \label{vaite}
K \Subset V_r \ \ {\text   and  } \ \   V_r \cap  \Gamma_\varepsilon = \emptyset.
\end{equation}
 Further,  the rational function $r=p/q$ can be so chosen that ${\rm deg }\ q < {\rm deg}\  p$.

\end{theorem}

   \begin{proof}
 We define a piecewise constant holomorphic function $\chi$  such that it vanishes in some small neighborhood of $K$ and  equals $2$ in a small neighborhood of $\Gamma_\varepsilon$.  Denoting $E=K \cup \Gamma_\varepsilon$  we can approximate $\chi$  by Runge's Theorem, [6],   with rational functions in $E$ uniformly. In particular there exists a rational function $r_0$ such that
 $$
 \max_{z\in E} | \chi(z) -r_0(z)| <1/2.
 $$
 Then in $K$ we have $|r_0(z)| <1/2$  while in $\Gamma_\varepsilon$ we have $|r_0(z)| > 3/2$.  Thus (\ref{vaite}) holds.

In order to show  that we can have  $r\rightarrow \infty $ as $z\rightarrow \infty$,  denote $P(z)=1-  (z / R)^n$ .  If $r_0$ is a rational function satisfying (\ref{vaite}) then with $R$ and $n$ large enough the rational function $r=P r_0$ still satisfies (\ref{vaite})  with ${\rm deg }\ q < {\rm deg}\  p$.  
   \end{proof}

 As $V_r$ consists of a finite number of components, bounded by  the degree of $p$  there is a finite number of components, each intersecting with $K$  and "surrounded" by $\Gamma_\varepsilon$.  Additionally $V_r$ may have components both "inside and outside" of $\Gamma_\varepsilon$.   Let us  denote by $V_r(K)$ the union of  the components of $V_r$ which intersect with $K$, so that in particular $K \Subset V_r(K)$.    
 Assume now that $\varphi$ is a holomorphic function in some neighborhood of $K$.  Then with small enough $\varepsilon$ there exists $r$ such that $\varphi$ is holomorphic in $V_r(K)$.  Denote by $\gamma$ the  boundary of $V_r(K)$, consisting of a finite number of piecewise smooth loops, and oriented so that $K$ stays on the left.  Then by Cauchy's theorem we have for $z\in K$
\begin{equation} 
\varphi(z)=\frac{1}{2\pi i} \int_\gamma \frac{\varphi(\lambda)}{\lambda-z} \ d\lambda.
\end{equation}
Observe that  along $\gamma$  we have $|r(z)|=1$ and thus $V_r(K)$ is mapped  in $w=r(z)$ onto the unit  disc -  and the scalar function $\varphi$: $ V_r(K) \rightarrow \mathbb C$ is likewise replaced by a vector-valued holomorphic function $f$:  $\overline {\mathbb D} \rightarrow \mathbb C^d$.  In order to achieve this, we shall decompose the Cauchy kernel  into pieces, each yielding one  component $f_i$ of  $f$.    Notice that $r^{-1} ( \overline {\mathbb D})=V_r$  may contain  components which  do not intersect $K$.  However,  we have the possibility to define $\varphi=0$ in those components and thus the integration and analysis could be done in the whole $V_r$ as well, if so wanted.

 \subsection{Spectrum as the compact set}
 
 Assume given a bounded operator  $A$  in a Banach space $\mathcal X$, $A\in B(\mathcal X)$.   Fix $\varepsilon >0$ and let $r=p/q$ be as in Theorem 2.1  when the  spectrum $\sigma(A)$ is taken as the compact set $K$.   In particular $r$ is holomorphic in the spectrum and $r(A)$ is a well defined bounded operator.   Then  
  $$
 \| r(A)^m\| ^{1/m} \rightarrow \rho(r(A)) = \sup_{z\in \sigma(A)} |r(z)|.
 $$
Since $r$ is not a constant, and $\sigma (A) \Subset V_r$,  we have by maximum principle $$ \sup_{z\in \sigma(A)} |r(z)| <1.$$  
But  then there exists $n$ such that $\|r(A)^n\| <1$.   Denote by $\tilde p$ a tiny pertubation of $p^n$ so that all roots of $\tilde r= \tilde p /q^n$ are simple  and we still have $\|\tilde r(A)\| <1$.    In order to formulate the corollary, let us denote 
 by  $\Gamma_\varepsilon$  the set  surrounding the spectrum as in (\ref{beltaround}) with $\sigma(A)=K$.  
 
 \begin{corollary}  Given a bounded operator $A\in B(\mathcal X)$,  fix  an $\varepsilon >0$ and denote  by $\Gamma_\varepsilon$ the set around the spectrum  $\sigma(A)$ as above.  Then there exists a rational function $r=p/q$, such that ${\rm deg }\ q < {\rm deg}\  p$  where $p$  has simple roots and  $\|r(A)\| <1$,  while $|r(z)| >1$  for $z\in \Gamma_\varepsilon$.
 \end{corollary}
 
A typical application of  using polynomials or rational functions as new variables is the possibility to deal with piecewise constant holomorphic functions.  We  mention two natural situations.  

\begin{example}   If the lemniscate set  covering the  spectrum has several components, then defining the holomorphic function to be identically 1 in one component  while setting it 0 in the others leads to an explicit power series representation for the Riesz spectral projection [2].   In Figure 1 we have a  model situation which cannot be obtained by polynomial lemniscates.  Two circles are separated from each others with a rational function  $r=p/q$ with  $p$ of degree 16 and $q$ of degree  9.  The  set in which    $|r(z)| <1$  is  white in the picture and dots denote the zeros of $p$  while small circles denote the zeros of $q$.  

\bigskip
\begin{figure}
\begin{center}
\includegraphics[scale=0.40]{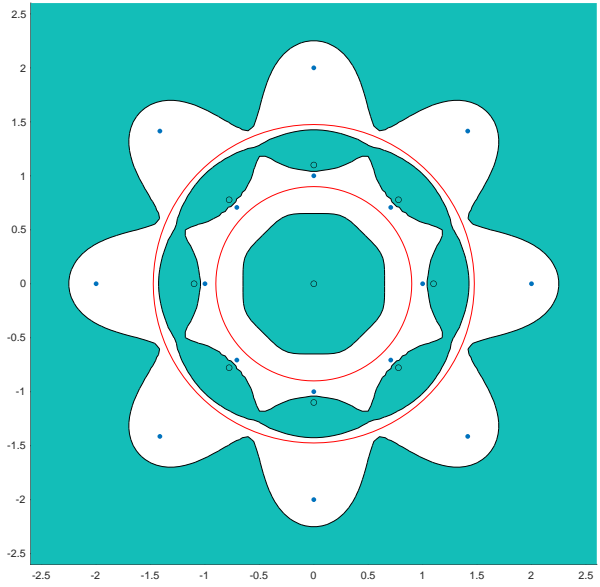}
\end{center}
\caption{A rational function of degree 16  separting two circles }
\end{figure}

 
\end{example}

\begin{example}  Another  natural piecewise holomorphic function is the sign-function
which equals 1 in the right half plane and $-1$ in the left.   In Section 4 we outline a use of it in the solving of the Sylvester equation.   Here one is after a polynomial or rational function such that the lemniscate  has components both on the left and right half planes without intersecting the imaginary axis. 
In [2]  sets consisting of two intervals, parallel to the imaginary axis and symmetrically located around the origin,  were considered as  test sets to be separated.  As the phenomenon is scaling invariant,  the angle $\alpha$  was used  to parametrize the sets $$L_\alpha = \{x + i y \ :  x \in \{-1,1\} , |y|  \le \tan (\alpha) \}.$$   Polynomials were then searched such that $L_\alpha \subset V_p$ while $V_p \cap i\mathbb R = \emptyset$. With $p(z)=z^2 -1$, suitably scaled, any angle below $45^{o}$ is clearly possible. With degree 4 one finds polynomials with angle  above $61^{o} $  but the required degree seemed to grow  quite fast with $\alpha$.
For example, angles above  $72^{o}$ were found only with polynomials of degree 14 or higher. For details , see [2].   As to be expected,  with rational functions the separation  is  easier and for example with  $d = 2$ and $r(z) = z -1/z$ the  largest angle, see Figure 3  below, is already about  $69^{o}$.  In order to have a  simple  rational function  of degree 4 consider
\begin{equation}\label{kokonaiskertoiminen}
r(z)= \frac{z^4 -2z^2+9}{z^3 +3z}
\end{equation}
which vanishes at $\pm \sqrt 2 \pm i$ and has poles at the origin and at $ \pm i \sqrt 3$.   In Figure 2  the lemniscate is drawn at the level $|r(z)|= 5.6$   with $\alpha=80^{o}$.  In  the Appendix it is demontrated   that the angle stays below $81^{o}$  for all rational functions  with $p$ of degree 4 and $q$ of degree 3.

\begin{figure}
\begin{center}
\includegraphics[scale=0.4]{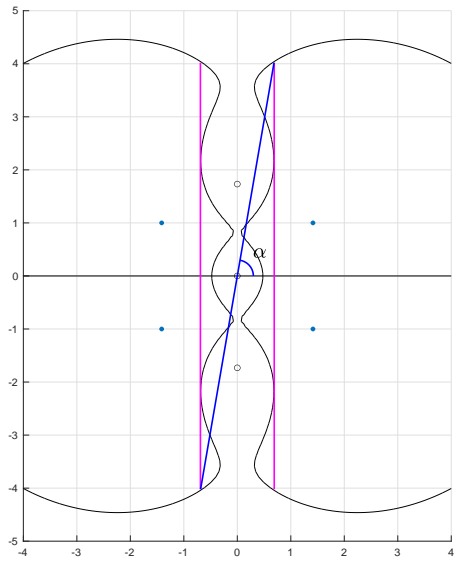}
\end{center}
\caption{The rational function  $r(x)$ in (\ref{kokonaiskertoiminen}) with $|r(z)|= 5.6$.}
\end{figure}


\end{example}

\section{Representation using rational functions as variables}

Let $r=p/q$ with $p$ having  $d$ simple roots  $\Lambda =\{\lambda_j\}$ so that  $q(\lambda_j)\not=0$  and such that  $d= {\rm deg} \  p > {\rm deg} \  q.$
  Denoting by $\delta_j$ the  rational functions   
\begin{equation}\label{perusmuoto}
\delta_j(z) = \frac{r(z)} {r'(\lambda_j) (z-\lambda_j)}
\end{equation}
we consider  representations  of scalar functions $\varphi$  in  the form
\begin{equation}\label{repr}
\varphi(z)=\sum_{j=1}^d \delta_j(z) f_j(w)   \ \text{ where } \  w= r(z).
\end{equation}

The assumptions that  the roots of $p$ to be simple and $q$ being  of lower degree than $p$ are not necessary but made for simplifying the discussion. 
 
\subsection{Existence and uniqueness}

We are interested in using $w=r(z)$ as a new  complex variable and assume in the following that $q(z)\not=0$.   Modifying the discussion in    section 2.1  in [7]  we take $w\in \mathbb C$ and denote by $z_j=z_j(w)$ the $d$ roots of 
\begin{equation}\label{creatingroots}
p(z) -  w \ q(z)   = 0.
\end{equation}
Let $z_j(w_0)$ be a simple root.  Then it is analytic at  $w=w_0$ with 
$$
z_j'(w_0) = \frac{q(z_j(w_0))}{p'(z_j(w_0)) -   w_0 \ q'(z_j(w_0)) }.
$$ 
 
Observe that since
$ r' = \frac{1}{q}(p' - r q')$,  all  finite critical values of $r$ agree with those of $p-w q$. 
So, let   $w_0\not= \infty$ be a noncritical value of $r(w)$ so that  the roots $z_j(w_0)$ are all distinct. 
Assuming that  {\it the values of $\varphi$ at these roots are all known},   we ask for "unknowns" $f_k(w_0)$ satisfying the  equations
\begin{equation}\label{pointwise}
\sum_{k=1}^d \delta_k(z_j(w_0)) f_k(w_0)= \varphi(z_j(w_0))
\end{equation}
for $j=1, \dots, d$.  We may write this as a linear system of equations 
\begin{equation}\label{pointmodified}
A(w_0) f(w_0) = \varphi(r^{-1}(w_0))
\end{equation} 
where $A(w)= (\delta_k(z_j(w)))_{j,k}$ is a square matrix, $f(w_0)\in \mathbb C^d$ has components $f_k(w_0)$ and $ \varphi(r^{-1}(w_0))\in \mathbb C^d$ has components $\varphi(z_j(w_0))$.
At noncritical $w_0$  the matrix $A(w_0)$ is nonsingular, as is easily seen by rewriting (\ref{pointwise})  as
$$
\sum_{k=1}^d \ell_k(z_j) x_k = (q \ \varphi)(z_j)
$$
where $z_j= z_j(w_0)$, $x_k= q(\lambda_k) f_k(w_0)$ and $\ell_k$ denote the Lagrange interpolation polynomials at $\lambda_k$
$$
\ell_k(z)= \frac{p(z)}{p'(\lambda_k) (z-\lambda_k)}.
$$

Assume now that $M \subset \mathbb C$ is compact and let $K = r^{-1} (M)$. Denote by $M_0 = M \setminus W_c$ where $W_c$ denotes the set of critical values of $r$ and put $K_0= r^{-1}(M_0)$.   

\begin{proposition}  Given a function $\varphi$ mapping $K_0 \rightarrow \mathbb C$, with $K_0$ as above, there exists a unique   $f$ mapping $M_0 \rightarrow \mathbb C^d$ such that
\begin{equation}
\varphi(z) = \sum_{k=1}^d \delta_k(z) f_k(r(z))
\end{equation} 
holds  for  $z\in K_0$.  
The function $f$ inherits the smoothness of  $\varphi$. In particular, if $\varphi $ is continuous  or  holomorphic  in $K_0$, then $f$ is continuous or holomorphic  in $M_0$.

\end{proposition}
\begin{proof}
Since $A(w)$ is holomorphic in $M_0$ and nonsingular,  then so is $A(w)^{-1}$.   The claims follow from
$$
f(w)= A(w)^{-1} \varphi(r^{-1}(w)).
$$
\end{proof}

At a critical value $w_c$ there are less equations  and $f_k(w_c)$'s do exist but are not unique.  
It is therefore of interest to study what continuity conditions on $\varphi$  guarantee  continuity of $f _k$'s at critical values.   We shall see, that  if $\varphi$ is holomorphic in $K$, then $f$ can be extended from $M_0$ to $M$ so that it is holomorphic  also at  the critical values.  We shall discuss this using Cauchy integral. However,  at this point it is natural to note,  that  the constant function $\varphi: z \mapsto 1$ is represented by $f: w \mapsto (1, \cdots, 1)^T$.
 
\begin{lemma} Let ${\rm deg} \ q <  {\rm deg} \ p$ and denote by $Z_q$ the  zeros of $q$.  Assume $z\notin Z_q$.  Then
\begin{equation}\label{sumto1}
\sum_{k=1}^d \delta_k(z) =1.
\end{equation}
\end{lemma}

\begin{proof}  For $q(z)\not=0$ we  have 
$
\sum_{k=1}^d \delta_k(z) = \frac{1}{q(z)} \sum_{k=1}^d 
q(\lambda_k) \ell_k(z).  
$ But the Lagrange interpolant  of $q$ equals $q$ as   ${\rm deg} \ q <  {\rm deg} \ p$.

\end{proof}

\subsection{Decomposing the Cauchy kernel}

Assume again  ${\rm deg} \ q <  {\rm deg} \ p$  and consider 
$$
 r[\lambda,z] = \frac{r(\lambda)-r(z)}{\lambda - z}. 
$$
Consider $\lambda$ to be fixed and such that $q(\lambda)\not=0$. Then $z\mapsto q(z) r[\lambda,z]$ is a polynomial of degree $ d-1$.  In fact,  it  is $\mathcal O (z^{d-1}) $  as $z \rightarrow \infty$ while  as $z\rightarrow \lambda$  it tends to $q(\lambda) r'(\lambda)$.    Hence  the Lagrange interpolation gives
$$
q(z)r[\lambda,z] = \sum_{j=1}^d \ell_j(z) q(\lambda_j)r[\lambda,\lambda_j].
$$
But since $r(\lambda_j)=0$ we can rewrite this  for $q(z)\not=0$ as
\begin{equation}\label{erranhajotus}
r[\lambda,z] = \sum_{j=1}^d \delta_j(z) \frac{r(\lambda)} {\lambda - \lambda_j}.
\end{equation}
Hence we have the following representation for the Cauchy kernel.
\begin{proposition} Let $r=p/q$ with ${\rm deg} \ q <  {\rm deg} \ p$.   Then 
\begin{equation}\label{cauchynhajotus}
\frac{1}{\lambda - z} = \sum_{j=1}^d \delta_j(z) K_j(\lambda, r(z))
\end{equation}
where
$$K_j(\lambda, w)= \frac{1}{\lambda - \lambda_j} \frac{r(\lambda)}{r(\lambda) -w}.
$$
\end{proposition} 
This allows us to conclude that if $\varphi$ is holomorphic in $\{z \ : \ |r(z)|  < \rho\}$ and continuous in $\{z \ : \ |r(z)|  \le\rho\}$, then $f_j$ is holomorphic in $|w| < \rho$.   To that end, denote by $\gamma_\rho$ the contour with points along $|r(\lambda)| = \rho$,   each finite curve oriented such that $|r(z)| <\rho$  stays on the left hand side.   We assume additionally {\it that $\gamma_\rho$  contains no critical points of $r$}, making the components  smooth.
Denote 
$$
L_{\rho, j} = \frac{1}{2\pi} \int_{\gamma_\rho} \frac{|d \lambda|}{|\lambda-\lambda_j|}.
$$

\begin{proposition}  Suppose $\varphi$ is holomorphic in $\lambda$ for  $|r(\lambda)| < \rho$ and continuous  in $|r(\lambda)| \le \rho$.  For $   |w| < \rho$ then 
\begin{equation}\label{integraaiesitys}
f_j(w) = \frac {1}{2 \pi i} \int_{\gamma_\rho} K_j(\lambda, w) \  \varphi(\lambda) \  d\lambda
\end{equation}
and $f_j$ is holomorphic in $|w|<\rho$ and can be expanded as a convergent series
$$
f_j(w) = \sum_{k=0}^\infty \alpha_{j,k} w^k ,  \ \ \text{ where } \ \  |\alpha_{j,k}| \le L_{\rho,j} \  \rho^{-k-1}  \  \max_{\lambda\in \gamma_\rho} |\varphi(\lambda)|.
$$
 
\end{proposition}

\begin{proof}
The contour $\gamma_\rho$ consists of a finite number of   smooth curves for which we  have for $z$ inside $\gamma_\rho$
$$
\varphi(z) = \frac {1}{2 \pi i} \int_{\gamma_\rho} \frac{\varphi(\lambda)}{\lambda - z} \  d\lambda.
$$
The claim follows substituting (\ref{cauchynhajotus}) into this.
\end{proof}

Notice in particular that  $f_j$ is holomorphic at critical values $|w_c| < \rho$.
 
We can localize this representation inside any number of components of $\gamma_\rho$.  In fact,   let $J \subset \{1,2,\cdots, d\}$ and let $\gamma_{\rho, J}$ consist of those components of $\gamma_\rho$  which surround at least one $\lambda_j$ with $j\in J$.    

\begin{corollary}  Assume $\varphi$ is holomorphic inside $\gamma_{\rho,J}$ and continuous up to  $\gamma_{\rho,J}$.  Then the previous proposition holds with $\gamma_\rho$ replaced  by $\gamma_{\rho, J}$.  In particular, if $j\notin J$, then $f_j=0$ and for $z$ inside  $\gamma_{\rho,J}$ we have 
$$
\varphi (z) = \sum_{j \in J} \delta_j(z) f_j(r(z)).
$$
\end{corollary}

\begin{proof}
We may define $\varphi=0$  along  the remaining components:  $\gamma_\rho \setminus \gamma_{\rho,J}$. 
\end{proof}

 \subsection{Derivative data at $\Lambda$}
 
Denote by $\partial$ the derivative w.r.t. $z$.  Then we have (Lemma 4.1 in [7])
$$
\varphi^{(\nu)} =  \sum_{k=1}^d \sum_{\mu=0}^\nu {\nu \choose \mu}
\delta_k^{(\nu - \mu)} \partial^\mu  (f_k \circ r).
$$
Proceeding as in the polynomial case, [7], it is easy to see that the formulas stay formally the same with $r'$  in place of $p'$.
Given the   values $\varphi^{(\nu)}(\lambda_j)$ we can compute  $f_j^{(\nu)}(0)$ from the following  
\begin{equation}
r'(\lambda_j))^\nu f_j^{(\nu)}(0)=  
 \varphi^{(\nu)} (\lambda_j) - h_{j,\nu}
 \end{equation}
 where
 $$
h_{j, \nu}= \sum_{k=1}^d \sum_{\mu=0}^{\nu-1}  {\nu\choose \mu} \delta_k^{(\nu-\mu)}(\lambda_j)\sum_{l=0}^\mu b_{\mu, l}(\lambda_j)
f_k^{(l)} (0)  
+ \sum_{l=0}^{\nu-1} b_{\nu,l}(\lambda_j) f_j^{(l)}(0).\footnote{the last term is missing in the corresponding line in [7]} 
$$
Then  the  power series
\begin{equation}\label{pow}
\sum_{\nu=0}^\infty \frac{f_j^{(\nu)}(0)}{\nu !}  w^\nu
\end{equation} 
represents $f_j$ in a disc with radius the same as the distance  from origin to the closest singularity of $f_j$.  
  
\subsection{Unital Banach algebra $C_\Sigma(M)$}

A simple functional calculus for diagonalizable matrices can be defined via similarity transformation into diagonal form.  If $A=T DT^{-1}$ then  one can define $\varphi(A)=T\varphi (D) T^{-1}$   with $\varphi (D)= {\rm diag} \ (\varphi (d_i))$.  
If $A$ has nontrivial Jordan blocks, then the following is possible:  take a "simplifying polynomial" $p$ with critical points with matching multiplicities at the eigenvalues corresponding to the nontrivial Jordan blocks.  Then $p(A)$ is diagonalizable, and again,  $\varphi (A)$ is well defined  via
$$
\varphi(A)=  \sum_{j=1}^d \ell_j(A) f_j(p(A)).
$$
In [9] this was approached as follows.  Consider the Banach space of continuous functions from a compact set $M$ into $\mathbb C^d$,  with max-norm. 
Then a "polyproduct"  $\circledcirc$  was constructed   such that  if $f$ represents $\varphi$ and $g$ represents $\psi$ then  $f\circledcirc g$ represents $\varphi \psi$:
$$
(\varphi \psi) (z) = \sum_{j=1} ^d \ell_j(z) (f\circledcirc g)_j(p(z)) 
$$
for $z\in p^{-1}(M)$.  

We indicate the key steps  as they go for the rational variable $w=r(z)$ in the same way.   To define the product, let $\{e_i\}$ denote the standard  basis of $\mathbb C^d$.   At this point we assume we are given a $ d \times d$  {\it multiplication table} $\Sigma=\{\sigma_{ij}\}$ and a frozen  $w \in M$.

\begin{definition}
Define in $\mathbb C^d$
$$
e_i \circledcirc e_i = e_i - w \sum_{j\not=i}(\sigma_{ij}e_i+\sigma_{ji}e_j)
$$and for $j\not=i$
$$
e_i \circledcirc e_j = w(\sigma_{ij}e_i+\sigma_{ji}e_j)
$$
and extend to $\mathbb C^d$ by linearity.
\end{definition}
The product is clearly commutative.  Denote ${\bf 1}= \sum_{i=1}^d e_i$.

\begin{lemma}
$$
{\bf 1} \circledcirc e_i=e_i.
$$
\end{lemma}

\begin{proof} We have using the definition
$$
{\bf 1} \circledcirc e_i= \sum_j e_j  \circledcirc e_i =   e_i \circledcirc e_i    + \sum_{j\not=i} e_j \circledcirc e_i
=e_i.
$$
\end{proof}

\bigskip
We shall now let $w\in M$ to vary, with $M\subset \mathbb C$ compact  and write the functions $f: M \rightarrow \mathbb C^d$ as 
$$
f:  w \mapsto \sum_{j=1}^d f_j(w) e_j
$$
and extend $\circledcirc $ to these  functions in a natural way by treating $f_j(w)$'s as scalars.   Passing to operator norm we obtain a {\it unital Banach algebra}.  Denote as before,  $|f|_\infty = \max_ {1\le j\le d} \max_{w\in M}|f_j(w)|$.

\begin{proposition} Defining in $C(M)^d$ 
$$
\| f\|  = \sup_{|g|_\infty \le 1}|f\circledcirc g|_\infty
$$
we have $|f|_\infty \le \|f\| \le C |f|_\infty$ with $C$  independent of $f$,  $\| {\bf 1} \|=1$  and $\| f \circledcirc g \| \le \|f \| \| g \|$.  

With this operator norm and  polyproduct $\circledcirc$ generated by the table $\Sigma$  we have a unital Banach algebra which we denote by $C_\Sigma(M)$. 

\end{proposition}
\begin{proof} These properties  hold in the similar way as in the polynomial case.
\end{proof}

For the full power of Banach algebra  machinery we need to know the set of characters.
Recall that a continuous  linear functional  $\phi$:  $C_\Sigma(M) \rightarrow \mathbb C$,  is a {\it character} if it is  nontrivial and multiplicative:
$$
\phi(a\circledcirc b)= \phi(a) \phi(b).
$$  
 It is well known that all characters in $C(M)$ are just evaluations $ \alpha  \mapsto \alpha(w_0)$, see e.g. [3].  We may identify the subalgebra of $C_\Sigma(M)$ consisting of elements of the form $w\mapsto \alpha(w) {\bf 1}$  with $C(M)$ and conclude that all characters in $C_\Sigma(M)$  reduce to  evaluations in this subalgebra.  Fix $w_0 \in M$ and let $\chi$ be any character   mapping $\alpha {\bf 1}\mapsto \alpha(w_0)$.  Then it follows that  for any $f \in C_\Sigma(M)$ the value $\chi(f)$ only depends on $f(w_0) \in \mathbb C^d$.  To see this, notice that for any $f$ we have $\alpha {\bf 1} \circledcirc f = \alpha f$ and  so $\chi(f)= \chi(\alpha f)$ provided $\alpha(w_0)=1$.  Let $\alpha_n (w) = 1- \min \{n \ |w-w_0|, 1\}$. Then $\chi(f-f(w_0))=0$. In fact, as $\chi $ is continuous in the operator norm, which is equivalent with the max-norm,
$$
|\chi(f-f(w_0))| = |\chi(\alpha_n(f-f(w_0))|  \le C \ \| \chi\|  \  |\alpha_n(f-f(w_0))|_\infty \rightarrow 0 
$$ 
as $f-f(w_0)$ is continuous.  Hence, $\chi(f)$   {\it only depends on the vector }$f(w_0)\in \mathbb C^d$.
Thus, $\chi$ acts as evaluation at $w_0$ followed by  a multiplicative functional  $\mathbb C^d \rightarrow \mathbb C$  with $\mathbb C^d$ equipped with the product $\circledcirc$,  where the variable $w$   takes the fixed value $w_0$.  But all linear functionals in $\mathbb C^d$ are of the form
$$
\eta:  x= \sum_{i=1}^d \xi_i  e_i \mapsto \sum_{i=1}^d \eta_i \xi_i.
$$
Requiring  $\eta({\bf 1}) = 1$ implies $\sum_{i=1}^d \eta_i=1$.   Consider first $w_0=0$.  Then  $\eta(e_i \circledcirc e_j)= 0$ implies  that $\eta$ has exactly one component size $1$ while the others vanish. Thus there are $d$ different characters.  Let then $w\not=0$.   From 
$$
\eta(e_i \circledcirc e_j) = \eta_i \eta_j = w_0 [\sigma_{ij} \eta_i + \sigma_{ji} \eta_j]
$$
we conclude that  $\eta_i \not=0$ for all $i$.   Applying to $e_i \circledcirc e_i$ we obtain
\begin{equation}\label{yhtalo}
\eta_i^2 = \eta_i - w_0 \sum_{j\not = i} [\sigma_{ij} \eta_i + \sigma_{ji} \eta_j
].
\end{equation}
Taking e.g. $\eta_i$ as an unknown, we can solve   $\eta_j$ for $j\not=i$  from $$ \eta_i \eta_j= w_0 [\sigma_{ij} \eta_i + \sigma_{ji} \eta_j],
$$
and substituting  them into  (\ref{yhtalo})  yields a polynomial equation for $\eta_i$ of degree $d$.  Thus, again there are (at most) $d$ characters for every $w_0$.    In general, the  components of characters depend on $w_0$ in  a rather complicated way. However,  when the multiplication table  $\Sigma$ is given by a rational function,  the dependence  can be explicitly given.

\begin{definition}\label{generated}   Let $p$ be monic of degree $d$ with simple roots $\{\lambda_j\}$ and $q$ of degree  at most $d-1$,  with $q(\lambda_j)\not=0$ and denote $r=p/q$.   If the multiplication table  satisfies
\begin{equation}
  \sigma_{ij}= \frac{1}{r'(\lambda_j)}\frac{1}{\lambda_i - \lambda_j},
\end{equation}
then we say that the product $\circledcirc$ in  $C_\Sigma(M)$  is determined by the rational function $r$.
\end{definition}

We shall next connect the products $e_i\circledcirc e_j$ to those of $\delta_i \delta_j$.  
\begin{lemma}    Assume that $\Sigma$ is determined by the rational function $r$. Then 
\begin{equation}\label{ratidentiteetti}
\delta_i^ 2= \delta_i -  \frac{p}{q} \sum_{j\not=i} [\sigma_{ij}\delta_i + \sigma_{ji}\delta_j] ,  \ \ \text{while  for } i \not=j , \ \  \delta_i\delta_j = \frac{p}{q}  [\sigma_{ij}\delta_i + \sigma_{ji}\delta_j].
\end{equation}
 
\end{lemma}
\begin{proof}  In the polynomial case with $q=1$ this is Lemma 1 in [9].  In fact,  we have in the polynomial case
\begin{equation}\label{polynomiidentiteetti}
 \ell_i^ 2= \ell_i -  {p} \sum_{j\not=i} [\tau_{ij}\ell_i + \tau_{ji}\ell_j] ,  \ \ \text{while  for } i \not=j , \ \  \ell_i\ell_j = {p} [\tau_{ij}\ell_i + \tau_{ji}\ell_j],
 \end{equation}
 where $\tau_{ij} = \frac{1}{p'(\lambda_j)}\frac{1}{\lambda_i - \lambda_j}$.
 Since  $\delta_i = \frac{q(\lambda_i)}{q} \ell_i$, $\sigma_{ij}Ê= q(\lambda_j) \tau_{ij}$ and $q= \sum_{i=1}^d q(\lambda_i) \ell_i$,  (\ref{ratidentiteetti})  follows from (\ref{polynomiidentiteetti}).
\end {proof} 

Assume now that $w$ is not critical, so that there are $d$ different roots $z_j=z_j(w)$ satisfying $p(z)-wq(z)=0$.  So, for each of these roots we have, with $f$ and $g$ representing $\varphi$  and $\psi$ respectively, that
$$
(\varphi \psi) (z_j(w))= \sum_{i=1}^d \delta_i(z_j) (f\circledcirc g)_i(w),
$$
which means that  for $z\in r^{-1}\{w\}$ and $w \in M$ 
\begin{equation}\label{karakteerimuoto}
\chi_{z}: f \mapsto \sum_{i=1}^d \delta_i(z ) f_i(w)
\end{equation}
is a character.   Observe that the character first evaluates $f$  at a point  followed by  application by functional  in $\mathbb C^d$ with components $\delta_i(z)$.   As different  roots $z_j(w)$  give $d$ different characters, all satisfying the polynomial equation for  the components  of $\eta$ discussed above,  we conlude that  we have found all characters. We may summarize  this in the following.

\begin{theorem}  In the unital commutative  Banach algebra $C_\Sigma(M)$ with $\Sigma$ generated by a rational function as in Definition \ref {generated}  all characters are of the form (\ref{karakteerimuoto}).   In particular 
the Gelfand transformation  $f \mapsto \hat f$ is given by $ \hat f(z)=\chi_z (f)$. \end{theorem}

Since  $\varphi (z) = \chi_z(f)$   the spectrum of $f$ is simply $\sigma(f)= \{ \varphi(z) \ :  z \in r^{-1}(M) \}.$ In [9]  the polynomial case is analysed in detail. For example,  if $| \varphi(z)| \ge \eta >0$ in $K= p^{-1}(M)$,  then there exists $g\in C_\Lambda(M)$ such that $g\circledcirc f = {\bf 1}$ satisfying
$$
\|g\| \le C \ \frac {\|f\|^{d-1}}{\eta^d},
$$
where the constant $C$ only depends on $M$ and on $p$.    Further, when applying the functional calculus  the set $M$ must contain $\sigma(p(A))$, but then  the inverse image $K= p^{-1}(M)$ may be essentially larger than $\sigma(A)$.  In such case   a quotient algebra appears useful.    
\begin{example}
As a simple rational function which is not a M\"obius transformation, consider
\begin{equation}
r(z)= \frac{1}{2}(z-\frac{1}{z}).
\end{equation}
We have $Z_p=\{1, -1\}$, $Z_q = \{0\}$, $r'(1)=r'(-1)=1$ , and thus  
$$
\delta_1(z)=\frac{1}{2}(1+\frac{1}{z}), \  \ \delta_2(z)=\frac{1}{2}(1-\frac{1}{z}), \ \  \sum_1^2 \delta_i(z) = 1.
$$
Further, with $\sigma_{i,j} = \frac{1}{r'(\lambda_j)}  \frac{1}{\lambda_i - \lambda_j} $
we have $
\sigma_{1,2}= \frac{1}{2}, \  \  \sigma_{2,1}= - \frac{1}{2}
$.
From 
$
z^2 - 2 w z -1=0$,   we obtain the inverse images $ z_\pm(w)= w \pm\sqrt{ 1+w^2}.
$
The 
critical points are at $z_c= \pm \  i $,  with critical values $w_c= \pm \ i$. 

 
\begin{figure}
\begin{center}
\includegraphics[scale=0.6]{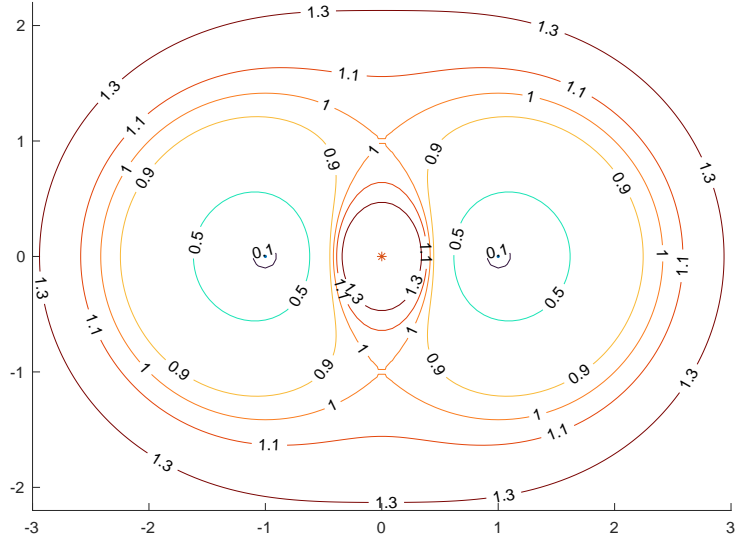}
\end{center}
\caption{Level curves of $z \mapsto  \frac{1}{2}  \ |z-1/z|.$}
\end{figure} 
 

 The  matrix $A(w)= \big( \delta_j(z_i(w))\big) $  mapping $f$ to  $\varphi$  in (\ref{pointmodified})  is 
 
 \begin{equation}
 A(w) = \frac{1}{2\sqrt{1+w^2}} \begin{pmatrix} 1+ \sqrt{1+w^2}-w& 1-\sqrt{1+w^2}+w\\
 1-\sqrt{1+w^2}-w&1+\sqrt{1+w^2}+w\end{pmatrix}    \end{equation}

  with inverse
   \begin{equation}\label{aaninverssi}
 A(w)^{-1} = \frac{1}{2\sqrt {1+ w^2}} \begin{pmatrix} 1+\sqrt{1+w^2}+w& -1+\sqrt{1+w^2}-w\\
 -1+\sqrt{1+w^2}+w&1+\sqrt{1+w^2}-w\end{pmatrix} . 
  \end{equation}
  
 For example, the variable $z$ is  represented  by
 $$
 f(w) = A(w)^{-1}\begin{pmatrix}  z_+(w) \\ z_-(w)\end{pmatrix} 
 = \begin{pmatrix}1+2w\\ -1+2w\end{pmatrix}.
 $$

\end{example}

\subsection{Relation between representations using $p$ and $r$}

As we may take both $p$ and $r=p/q$ as  new variables there  naturally exists a mapping between the representations.  In fact, let $\varphi$ be holomorphic in all that comes and suppose the  multicentric representation using polynomial variable $w=p(z)$ is  denoted as
$$
\varphi(z) = \sum_{j=1}^d \ell_j(z) F_j(p(z))
$$
and 
$$
q(z)=  \sum_{j=1}^d \ell_j(z) \ Q_j(p(z))
$$
where $Q_j(w)= q(\lambda_j)$.  Then 
\begin{equation}
(\varphi q) (z) = \sum_{j=1}^d \ell_j(z) (F \circledcirc Q)_j  (p(z))
\end{equation}
Using $r=p/q$ as the new variable we have
$$
\varphi (z)= \sum_{j=1}^d \delta_j(z) f_j(r(z))
$$
from which  we obtain
\begin{equation}
(\varphi q ) (z) = \sum_{j=1}^d \ell_j(z) q(\lambda_j) f_j(r(z)).
\end{equation}
Hence we have  
\begin{equation}
(F\circledcirc Q)(p(z))= (f\circ Q)(\frac{p}{q}(z)),
\end{equation}
where we denote by $\circ$ the elementwise product.
\bigskip
Finally, if we write $Q^{-1}$ for the vector  with components $1/q(\lambda_j)$ we have
 $$
 f(\frac{p}{q}(z)) = (Q^{-1}\circ F\circledcirc Q)(p(z)).
 $$

\section{Application to Sylvester equation}

Let $A$ and $B$ be bounded operators in Banach spaces $\mathcal X$, $\mathcal Y$ respectively.  Then  $AX-XB = C$ is the related Sylvester equation, where $C$ is a given operator and $X$ the unknown, both mapping  $\mathcal Y$ to $\mathcal X$.  It is well known [4] that a unique bounded  $X$ exists   for  every bounded $C$ if and only if $A$ and $B$ have disjoint spectra:  $\sigma(A) \cap \sigma(B) = \emptyset.$  In  Section 5 of [10] the   multicentric calculus was applied to the case where  the polynomial convex hulls of the spectra were disjoint:  $\widehat{\sigma(A)} \cap \widehat{\sigma(B)}= \emptyset$.  The solution was constructed as a convergent power series  provided that one has a polynomial lemniscate which separates the spectra into different components.   Here we outline the approach using rational lemniscates, which then  removes the need to assume that the polynomially convex hulls  do not intersect.  

Denote 
\begin{equation}\label{emm}
M= \begin{pmatrix} A&C\\&B\end{pmatrix}.
\end{equation} 
    
Observe that
 \begin{equation}\label{triang}
 M =\begin{pmatrix} I&-X\\
&I\end{pmatrix}\begin{pmatrix} A\\
&B\end{pmatrix}\begin{pmatrix} I&X\\
&I\end{pmatrix}
\end{equation}
is satisfied exactly when $AX-XB=C$. 
Denote  by $\rm{sgn} (z)$  the  function taking value 1 in the open right half plane $\mathbb C_+$  and $-1$ in the left one. 
If $\sigma(A) \subset \mathbb C_{+}$ and $\sigma(B)\subset \mathbb C_{-}$,    then $\rm{sgn}(M)$ is well defined    and we have 
\begin{equation}\label{signinkautta}
{\rm sgn}(M) =\begin{pmatrix} I&-X\\
&I\end{pmatrix}\begin{pmatrix} I \\
&-I\end{pmatrix}\begin{pmatrix} I&X\\
&I\end{pmatrix}=\begin{pmatrix} I&2X\\
&-I\end{pmatrix}.
\end{equation}
Thus,  $X$ can be obtained if  ${\rm sgn}(M)$ can be computed, see e.g. [4],[5]. 

Assume now only that $\sigma(A) \cap \sigma(B)=\emptyset$.  As the spectra are compact sets, there exist open $U_i$ such  that $\sigma(A) \subset U_1$, $\sigma(B) \subset U_2$ and  $U_1 \cap U_2=\emptyset$.   
Let the $\gamma_1$ be a contour surrounding $\sigma(A)$ inside $U_1$.   Then 
denoting 
\begin{equation}\label{melkein}
Q= \frac{1}{2\pi i}\int_{\gamma_1} (\lambda-M)^{-1}
\end{equation} 
we have
$$
Q=  \begin{pmatrix} I&-X\\
&I\end{pmatrix}\begin{pmatrix}I&0 \\
0&0\end{pmatrix}\begin{pmatrix} I&X\\
&I\end{pmatrix}=\begin{pmatrix} I&X\\&0\end{pmatrix}.
$$
Now define, in  place of the sign-function, $\psi(z) = 1$ for $z\in U_1$   while $\psi(z) =-1$  for $z\in U_2$.
In order to have  a convergent series expansion for $X$ let $\gamma_2$ be a contour surrounding $\sigma(B)$ inside $U_2$  and denote $\gamma= \gamma_1 \cup \gamma_2$. Thus 
$$
I = \frac{1}{2\pi i}\int_{\gamma} (\lambda-M)^{-1}.
$$ 
But then  adding this to  both sides of 
$$
\psi(M)= Q - \frac{1}{2\pi i}\int_{\gamma_2} (\lambda-M)^{-1}
$$
yields 
$
\psi(M) = 2Q-I
$
and $Q=\frac{1}{2} ( \psi(M) +I).$ 

Hence we have reduced the solving of the Sylvester equation into computing $\psi(M)$.  In order to do  that we need a rational function which separates the spectra of $A$ and $B$.

\begin{proposition} Let $M$ in {\rm (\ref{emm}) } be  given  and such that
$\sigma(A) \cap\sigma(B)=\emptyset$.  Let $\varepsilon >0$ satisfy $\varepsilon < {\rm dist}(\sigma(A), \sigma(B))/2$.  Then  there exists a rational function $r=p/q$ such that $p$ has distinct roots,  ${\rm deg}\ q < {\rm deg}\ p$  and such that the components of  the lemniscate set 
\begin{equation}\label{ratinkluusio}
V= \{z  \ : |r(z)|  <1  \}
\end{equation}
can be grouped into three disjoint sets: $V=V_1  \cup V_2 \cup V_0$ where $\sigma(A) \Subset V_1$, $\sigma(B) \Subset V_2$,   while $V_0$, which  may empty,  satisfies  $V_0 \cap \sigma(M) = \emptyset$.  Further, ${\rm dist} (V_i, V_j)\ge \varepsilon/2$ for $i\not=j$.
 \end{proposition} 
 
\begin{proof}
This   follows from Corollary 2.2. by applying it to the operator $M$.    In fact, there exists a rational  function $r$  such that 
 $\|r(M)\| <1$ while  for $z\in \Gamma_\varepsilon$  we have $|r(z)| >1$.   We have
 $$
 \sigma(M) = \sigma(A) \cup \sigma(B) \subset \{z \ :  \ |r(z)| \le \|r(M)\| \} \Subset V.
 $$ 
Then  collect all components of $V$ for which the distance to $\sigma(A)$ is at most $\varepsilon/2$  into $V_1$, those which are likewise close to $\sigma(B)$ into $V_2$  and the rest, if any, into $V_0$.  Now $V_1$ and $V_2$  are surrounded by  $\Gamma_\varepsilon$ of width $\varepsilon/2$,  and the claims follow. 
 
 \end{proof}

Assume now that $r$ satisfies the assumptions of the previous proposition.  We set $\psi(z)=1$ in $ V_1$, $\psi(z)=-1$ in $V_2$ and $\psi(z)=0$ in $V_0$.  Then we have   $$
\psi(z)=\sum_{j=1} ^d \delta_j(z) h_j(r(z))
$$
where $h_j$'s are holomorphic  for $|w|  < 1$.  Computing the  power series  $h_j(w)= \sum_{k=0}^\infty h_{j,k}w^k$  then gives  an explicit   expression for $\psi(M)$.   
 
 \begin{proposition}  Under the notation and assumptions above,  the  upper right corner element  of $\psi(M)$ is $2X$ where $X$ is the solution of  the Sylvester equation $AX-XB=C$.  
 \end{proposition}   
 
Notice that the power series converges for $|w| <1$.  When  $ \| r(M)\| <1$  one can truncate the power series with  the  possibility to  bound the truncation error.   Notice that  the  asymptotic convergence factor $\eta$  is given  by
$$
\eta = \max _ {\lambda \in \sigma(M) } |r(\lambda)|
$$
and is hence independent of the "right hand side"  $C$.

\begin{example}
Let again $r(z)=(z-1/z)/2$ and 
$
M=\begin{pmatrix} a&c\\ &b \end{pmatrix}
$
 be such that $\sqrt 2-1 < a, - b < \sqrt 2 +1 $, so that  $|r(a)|, |r(b)| <1$.  Hence, this serves as miniature model for  solving the "Sylvester equation"   $a\xi -\xi b =c$ along  the lines above.   The set where $|r(z)|<1$ has two components, one in the right half plane and the other in the left.  Thus choosing $\psi$ to take the value 1 in the  component  with $ {\rm Re} \ z >0$  and  $-1$ in the other one we actually  arrive into the restriction of  sign-function into these sets.  This however just follows from the  simple  form of $r$.  So, the answer shall  be  
 $$
{\psi}  (M)= \begin{pmatrix} 1& \frac{2c}{a-b}\\ &-1Ê\end{pmatrix}
$$
with $2 \xi$ appearing  in the upper right hand corner but we proceed without knowing the simple answer. Thus, we need to have  $f$ representing this $\psi$ in the unit circle $|w|<1$ and this is given immediately  from (\ref{aaninverssi})  
$$
f(w)= A(w)^{-1} \begin{pmatrix}1\\-1Ê\end{pmatrix} = \frac{1}{\sqrt{1+w^2}}\begin{pmatrix} 1+w\\-1+w\end{pmatrix}. 
$$
Hence
 for $z$ in $|r(z)|<1$ we have
 $
 \psi(z)= \delta_1(z) f_1(r(z)) + \delta_2(z) f_2(r(z))
 $ which simplifies into
 \begin{equation}\label{psiinesistys}
 \psi(z) = t(z) (t(z)^2)^{-1/2},  \  {\text where } \ \ \ t(z)= \frac{1}{2}(z+\frac{1}{z}).
 \end{equation}
 Since $t(z)^2 = 1+ r(z)^2$ we arrive to an explicit series expansion for $\psi$:
 \begin{equation} 
 \psi(z) = t(z) \ \big( 1-  \frac{1}{2} r(z)^2 + \frac{3}{8} r(z)^4  - \cdots\big).
 \end{equation}
 Notice that if $ |a-1| ,  |b+1| \le \varepsilon$, then  the spectral radius of $r(M)^2$ satisfies $\rho(r(M)^2) < \varepsilon^2$ and the convergence of the series  for $ (t(M)^2)^{-1/2}$ would be rapid and truncation  could be done safely.   The situation would remain  similar if the scalars $a$ and $b$ would be replaced with  bounded operators $A$ and $B$ with spectra near 1 and $-1$, respectively.
\end{example}

\section{Application to K-spectral sets}

It was  shown in [8]    that  polynomial lemniscate sets   are K-spectral sets, provided that the  boundaries are smooth, i.e.    do not contain critical points.   Here we point out that this extends   to rational lemniscates.    The proof in [8] goes as follows.  Representing the holomorphic function in the multicentric form  leads  us to estimate the components of $f$ evaluated at $p(A)$.  But  since  $f$  maps in a disc,  we can apply the von Neumann inequality to  get $\|f_j(p(A))\| \le |f_j|_\infty$.   The third step  needed, is to bound $f$ in terms of $\varphi$.    To repeat this in  the rational lemniscate case,  we formulate the last step  in the following lemma. 

\begin{lemma} Suppose  $r=p/q$ where $p$ has simple roots $\lambda_j$, $q(\lambda_j)\not=0$ and ${\rm deg} \ q < {\rm deg} \ p= d$. Suppose $R$ is such that  $|r(z)| = R$  contains no critical points of $r$.   Then there exists a constant $C(r,R)$ such that for all  $f$  with  components $f_j$ holomorphic in $|w| \le R$ there holds
$$
\sup_{|w| \le R} |f(w)|_\infty  \le C(r,R) \sup_{|r(z)| \le R} |\varphi(z)|.
$$
\end{lemma} 
\begin{proof}  This follows from the Cauchy integral formulation.
\end{proof}

Then we have the following.

\begin {theorem}  Let $r$ and $R$ be as in the previous lemma.  Suppose $A$ is  a bounded operator in a Hilbert space such that $||r(A)|| \le R$.   If $\varphi$ is holomorphic in $V_r(R)=\{ z : |r(z)| \le R\}$, then 
\begin{equation}\label{paras}
\| \varphi(A)\|   \le K \sup_{V_r(R)} |\varphi|
\end{equation}
where $K =  C(r,R) \sum_ {j=1}^d \|\delta_j(A)\|$.
\end{theorem}
\begin{proof}
In
$$
\|\varphi(A)\| \le \sum_{j=1}^d \|\delta_j(A)\|  \ \|f_j(r(A))\|  
$$
 we apply the von Neumann inequality to get  $\|f_j(r(A))\| \le \sup_{|w| \le R} |f_j(w)|$ and then  bound these by $\sup_{V_r(R)} |\varphi|$ using  Lemma 5.1.
 \end{proof}
 
 Thus,  the sets are K-spectral sets with constant independent of the holomorphic function $\varphi$ but depending on the geometry of the set and on $A$ through  $\delta_j(A)$.

\begin{remark}
The constant $C(r,R)$  depends on the distance  from the lemniscate to  critical points ns ia independent of the operator $A$.   In  [8] it  is shown that  in the polynomial case we have
\begin{equation}\label{vNex}
 C(p,R) \le 1 + \frac{C} {s(R)^{d-1}}
 \end{equation}
 where $s(R)$ denotes the distance to the nearest critical point.  Generically  the behavior is proportional to  $1/s(R)$ but  higher powers occur with possible multiplicities of the critical points.   Example  2.4 in [8] shows that the  worst case in (\ref{vNex}) can happen.   Recall that we denoted by $A(w)$ the matrix mapping $f$ to  $\varphi$, see (\ref{pointmodified}).    If $R$ is small enough  so that all critical values $w_c$ satisfy $|w_c| >R$, then $A(w)^{-1}$ is holomorphic for $|w| \le R$  and  we have
 $$
 C(r,R) \le \sup_ {|w| \le R} \| A(w)^{-1}\|_\infty
 $$
 where we denote by $\| . \|_\infty$ the  matrix norm induced by the $\max$ - norm in $\mathbb C^d$. As the  growth  exponent  in $C(r, R)$   when $s(R) \rightarrow 0$ depends on the multiplicity of the critical points and this behavior is local in nature,  we shall not repeat the argument as  it goes in the same way as in the polynomial case.   Rather, we again  return to  the simple $d=2$ case with $r(z) = (z - 1/z)/2$. 
 
\end{remark}

\begin{example}
The mapping matrix $A(w)$ for $r(z)= (z -1/z)/2$  has the inverse given in (\ref{aaninverssi}).  When $R =1-\varepsilon$  and $\varepsilon \rightarrow 0$  the distance  $s(R)$  behaves like $(1+o(1)) \ \varepsilon$  and 
$$
\| A(w)^{-1}\|_ \infty = 1 + (1+o(1)) \ \varepsilon ^{-1}. 
$$

\end{example}

\begin {example}
 The other  consider again the rational function $r(z)= (z-1/z)/2$  together with the matrix

$$
A= \begin{pmatrix} 1&c\\
0&-1\end{pmatrix}.
$$
 Since $r(A)=0$ the  coefficient  $C(r, R)$ in Lemma 5.1   shrinks to $C(r,0)=1$ and  (\ref{paras})  holds with $K= \sum_1^2 \| \delta_j(A)\|$.  With  $\delta_1(z) = (1 +1/z)/2$ and  $\delta_2(z)=(1-1/z)/2$  we have 
 
$$
\delta_1(A)=\begin{pmatrix} 1& c/2\\
0& 0\end{pmatrix},  \ \delta_2(A)=\begin{pmatrix} 0& - c/2\\
0& 1\end{pmatrix}.
$$
For example,  the Riesz projection wrt  to  the eigenvalue $\lambda_1=1$ is $\delta_1(A)$  while with $\varphi(z)=z$ we  have

$$
A = \delta_1(A) \cdot  1 + \delta_2(A) \cdot (-1)
$$
which shows that $K= \sum_1^2 \| \delta_j(A)\|$ becomes tight when $|c|$ grows.   Finally, notice that $A$ is of the form 
$$
M= \begin{pmatrix} a&c\\
0&b\end{pmatrix}
$$
for which the corresponding "Sylvester equation" reads  $ax-xb =c$  with solution 
$x=c/(a-b)=c/2$ to be found in the upper right  corners of $\delta_1(A)$ and  $-\delta_2(A)$.

\end{example}

{\bf References}

\bigskip

[1]  Diana Andrei, Multicentric holomorphic calculus for n-tuples of commuting operators, Adv. Oper. Theory, Vol. 4, Number 2 (2019), 447-461

[2]  Apetrei, Diana, Nevanlinna, Olavi: Multicentric calculus and the Riesz projection, Journal of Numerical Analysis and Approximation Theory. 44 (2), 2016, p. 127-145 .

[3]  B. Aupetit, A Primer on Spectral Theory, Springer  1991

[4] R. Bhatia, P. Rosenthal, How and Why to Solve the Operator Equation AX - XB = Y,  Bull. London Math. Soc., 29  (1997)1 - 21

[5]   N. J. Higham. Functions of Matrices. Society for Industrial and Applied Mathematics
(SIAM), Philadelphia, PA, (2008)

[6]   D. Gaier, Lectures on Complex Approximation,  Birkh\"auser, 1985

[7] O. Nevanlinna, Multicentric Holomorphic Calculus, Computational Methods and Function Theory, June 2012, Vol. 12, Issue 1, 45 - 65.

[8] O. Nevanlinna, Lemniscates and K-spectral sets, J. Funct. Anal. 262, (2012), 1728 - 1741.

[9] O. Nevanlinna, Polynomial as a New Variable - a Banach Algebra with Functional Calculus, Oper. and Matrices 10 (3) (2016)  567 - 592

[10] O. Nevanlinna, Sylvester equations and polynomial separation of spectra, Oper.  and Matrices 13, (3) (2019),  867-885

[11]   T. Ransford, Potential Theory  in the Complex Plane, London Math. Soc. Student Texts {\bf 28}, Cambridge Univ. Press, 1995

\newpage

{\bf APPENDIX}

Consider the separation of vertical lines by rational functions  with $p$ of degree 4 and $q$ of degree 3.  For  reasons of symmetry and scaling invariance we look at rationals functions with zeros at $\pm a \pm i$ and with poles at the origin and at $\pm b$, see Figure 4.


\begin{figure}
\begin{center}
\includegraphics[scale=0.5]{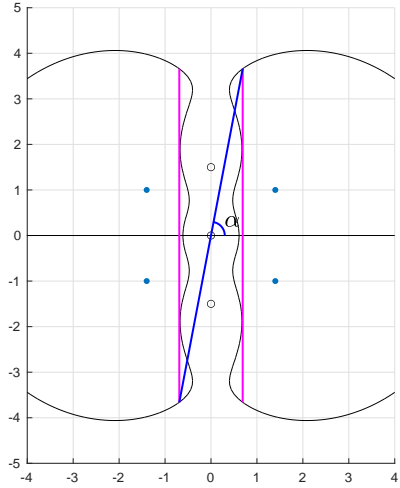}
\end{center}
\caption{Lemniscate  $|r(z)|=5.1$  where   $a=1.4$ and $b=1.5$, $x=0.69$, $y=3.66$ }
\end{figure} 


 
In order to approximate  the largest possible angle a numerical search was done by  numerically computing the  supremum level of $|r(z)|=R$ for each parameter pair ${a,b}$  with as large as possible ratio $y/x = \tan {\alpha}$.  These maximizing ratios are shown in Figure 5 with $a$ on the horizontal axis for each fixed $b$,  the enveloping curve being quite flat  between 1.2 and 1.5.  The  corresponding angles stay below $81^{o}$. In particular, the rational function in Example 2.4 with $a= \sqrt 2$ and $b= \sqrt 3$ is   nearly as good as the highest ones. 


\begin{figure}
\begin{center}
\includegraphics[scale=0.45]{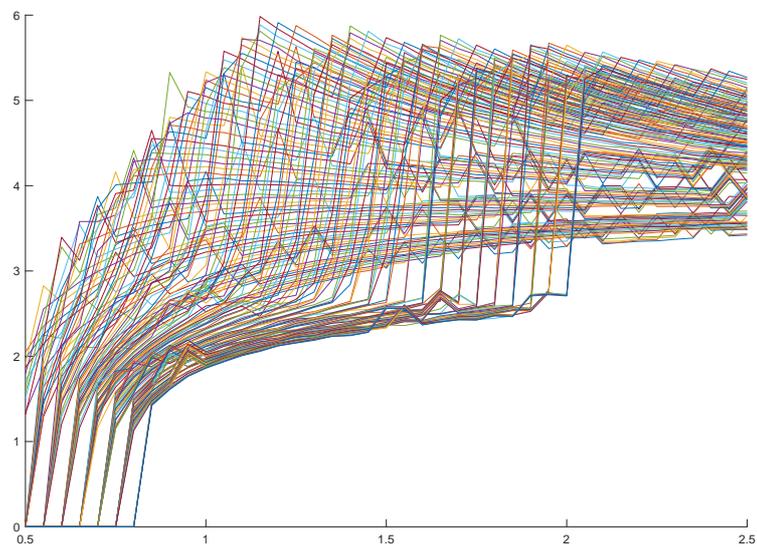}
\end{center}
\caption{Maximal ratios $y/x$   as  functions of $a$ for fixed $b$'s.}
\end{figure} 


\end{document}